\documentclass[12pt]{amsart}
\usepackage{color}
\usepackage{amssymb}
\usepackage{comment}
\usepackage{pdfpages}
\usepackage{graphicx}
\usepackage{dcolumn}

\RequirePackage[numbers]{natbib}
\RequirePackage[colorlinks=true, pdfstartview=FitV, linkcolor=blue,
  citecolor=blue, urlcolor=blue]{hyperref}
\RequirePackage{hypernat}
\usepackage{paralist}

\usepackage{graphicx}
\graphicspath{{./figures/}}
\usepackage{amsfonts}
\usepackage{amsmath}
\usepackage{amsthm}
\usepackage{amssymb}
\usepackage{amsbsy}
\usepackage{epsfig}
\usepackage{fullpage}
\usepackage{natbib, mathrsfs} 
\usepackage{verbatim}
\usepackage[latin1]{inputenc}
\usepackage{mhequ}
\usepackage{algorithm}
\usepackage{algorithmic}

\numberwithin{equation}{section}

\def \be{\begin{equs}}
\def \ee{\end{equs}}
\def \P{\mathbb{P}}
\def \E{\mathbb{E}}

\newcommand \TV{\mathrm{TV}}

\def \tmix{\tau_{\mathrm{mix}}}

\def \TV{\mathrm{TV}}

\def \sn{\mathrm{S}^{n-1}}
\def \F{\mathrm{F}}

\def \S{\mathrm{S}}

\newtheorem{theorem}{Theorem}[section]
\newtheorem{lemma}[theorem]{Lemma}

\theoremstyle{plain}
\newtheorem{thm}{Theorem}
\newtheorem*{thm-non}{Theorem}

\theoremstyle{definition}

\newtheorem{defn}[theorem]{Definition}
\newtheorem{remark}[theorem]{Remark}

\begin{document}

\title[Kac's Walk on $n$-sphere mixes in $n\log n$ steps]
{Kac's Walk on the $n$-sphere mixes in $n\log n$ steps}


\author{Natesh S. Pillai$^{\ddag}$}
\thanks{$^{\ddag}$pillai@fas.harvard.edu, 
   Department of Statistics
    Harvard University, 1 Oxford Street, Cambridge
    MA 02138, USA}

\author{Aaron Smith$^{\sharp}$}
\thanks{$^{\sharp}$smith.aaron.matthew@gmail.com, 
   Department of Mathematics and Statistics
University of Ottawa, 585 King Edward Drive, Ottawa
ON K1N 7N5, Canada}

\maketitle






\begin{abstract}
Determining the mixing time of Kac's random walk on the sphere $\sn$ 
is a long-standing open problem. We show that the total variation mixing time of Kac's walk on $\sn$ is between $\frac{1}{2}n \log(n)$ and $200 \,n \log(n)$ for all $n$ sufficiently large. Our bound is thus optimal up to a constant factor, improving on the best-known upper bound of $O(n^{5} \log(n)^{2})$ due to Jiang \cite{Jian11}.  Our main tool is a `non-Markovian' coupling recently introduced by the second author in \cite{Smit11a} for obtaining the convergence rates of certain high dimensional Gibbs samplers in continuous state spaces.
\end{abstract}

\section{Introduction}
In his 1954 paper \cite{Kac1954}, Marc Kac introduced a random walk on a sphere as a model for a Boltzmann gas. Kac's walk on the $(n-1)$-sphere,  
\be
\sn = \{ X \in \mathbb{R}^{n} \, : \, \sum_{i=1}^{n} X[i]^{2} = 1 \},
\ee
is a discrete-time Markov chain $\{X_{t} \in \sn\}_{t \geq 0}$  that evolves as follows. At every step $t$, choose two coordinates $1 \leq i_{t} < j_{t} \leq n$ and an angle $\theta_{t} \in [0, 2 \pi )$ uniformly at random and set
\be [KacStepRep]
 \left( \begin{array}{c}
X_{t+1}[i_t]  \\
X_{t+1}[j_t]
\end{array} \right) &=  \left[ \begin{array}{cc}
\cos(\theta_t) & -\sin(\theta_t) \\
\sin(\theta_t) & \cos(\theta_t) 
\end{array} \right] \left( \begin{array}{c}
X_t[i_t]  \\
X_t[j_t]
\end{array} \right), \\
X_{t+1}[k] &= X_{t}[k], \qquad k \notin \{i_{t}, j_{t} \}.
\ee
Let $\F: \{1,2,\dots,n\} \times \{1,2,\dots,n\} \times [0,2\pi)\times \sn \mapsto \sn$ be the map associated with this representation, so that $X_{t+1} = \F(i_{t},j_{t},\theta_{t},X_{t})$. Physically, Kac motivated this random walk by considering the velocities of $n$ particles in a one-dimensional box. He assumed that these particles were uniformly distributed in space, and the random walk $X_t$ models the change in their velocities over time as collisions occur. The condition that $X_t$ be constrained to a sphere corresponds to the principle of conservation of energy. Understanding the mixing properties of this process is central to Kac's program in kinetic theory.
The article \cite{mischler2013kac} gives a useful description of this program.

To state the main result of this paper, we recall some standard definitions. For measures $\nu_1, \nu_2$ on a measure space $(\Omega, \mathcal{F})$, the \textit{total variation distance} between $\nu_1, \nu_2$ is given by
\be 
\| \nu_1 - \nu_2 \|_{\TV} = \sup_{A \in \mathcal{F}} (\nu_1(A) - \nu_2(A)).
\ee 
We denote the distribution of a random variable $Z$ by $\mathcal{L}(Z)$ and write $Z \sim \nu$ as a shorthand for $\mathcal{L}(Z) = \nu$. \par

Let $\mu$ denote the normalized Haar measure on $\sn$. The Markov chain $\{X_{t} \}_{t \geq 0}$ has $\mu$ as its unique stationary distribution on $\sn$. The \textit{mixing profile} of $X_t$ is defined as 
\be 
\tau(\epsilon) = \min \{t \, : \, \sup_{X_{0} = x \in \sn} \| \mathcal{L}(X_{t}) - \mu \|_{\TV} < \epsilon \} 
\ee 
and the \textit{mixing time} is given by $\tmix = \tau(0.25)$.  

Our main result is:

\begin{thm}\label{MainThm}
Fix $C_{1} <\frac{1}{2}$ and $C_{2} > 200$. If the sequence of times $\{T_1(n)\}_{n\in\mathbb{N}}$ satisfies $T_1(n) < C_{1} n \log(n)$ for all $n$, then 
\be \label{EqMainThmLowerBound} 
\lim_{n \rightarrow \infty} \inf_{X_{0} \in \sn} \| \mathcal{L}(X_{T_1(n)}) - \mu \|_{\TV} = 1.
\ee 
If the sequence of times $\{T_2(n)\}_{n\in\mathbb{N}}$ satisfies $T_2(n) > C_{2} n \log(n)$ for all $n$, then 
\be \label{EqMainThmUpperBound}
\lim_{n \rightarrow \infty} \sup_{X_{0} \in \sn}  \| \mathcal{L}(X_{T_2(n)}) - \mu \|_{\TV} = 0.
\ee 
\end{thm}

Theorem \ref{MainThm} implies that, for any $\delta > 0$ and $n > N_{0}(\delta)$ sufficiently large, the mixing time $\tmix$ of Kac's walk on $\sn$ satisfies 
\be 
(1-\delta)\frac{1}{2} n \log(n) \leq \tmix \leq (1+\delta)\, 200\, n \log(n).
\ee 
Theorem \ref{MainThm} also establishes the stronger result that Kac's walk exhibits \textit{pre-cutoff} with window $[\frac{1}{2} n \log(n),200 \,n \log(n)]$ (see \textit{e.g.}, chapter 18 of \cite{LPW09} for an introduction to cut-off). We conjecture that Kac's walk also exhibits cutoff at time $2 n \log(n)$. While we have made no effort to optimize the size of the cutoff window obtained in this paper and both our upper and lower bounds for $\tmix$ could be easily improved with our methods, a proof of cut-off does not seem to follow from our methods and will likely require new ideas.

\subsection{Relationship to Prior Work}
There is an extensive literature on the mixing time of Kac's walk under various measures. In the sequence of papers \cite{jan2001, carlen2003, maslen2003, caputo2008spectral}, the spectrum of Kac's walk on $\S^{n}$ was bounded and then computed. Although these results imply a convergence \textit{rate} for Kac's walk in $L^{2}$, and a bound on the \textit{distance} to stationarity in $L^{2}$ implies a bound on the total variance distance to stationarity, these bounds \textit{do not} imply any bound at all on the total variation mixing time of Kac's walk. This is because, when $\mathcal{L}(X_{0})$ is concentrated at a point, the initial $L^{2}$ distance to stationarity is not finite. Stronger entropy bounds \cite{carlen2008entropy} also give no bound on the total variation mixing of Kac's walk. The previous best known bound on the mixing time of Kac's walk on $\S^n$, due to Jiang \cite{Jian11}, is $O(n^{5} \log(n)^{2})$. \par
There are other versions of Kac's walk; see the literature review in \cite{Oliv07} for an overview of one of the most-studied such walks. The recent works \cite{mischler2013kac,hauray2014kac} discuss the relationship of the study of Kac's walk to the original goals of Kac's paper \cite{Kac1954}.

\section{Technical Overview} 
We bound the mixing time of a Markov chain by applying the standard \textit{coupling lemma} (see \textit{e.g.}, Theorem 5.2 of \cite{LPW09}):
\begin{lemma}\label{LemmaIneqCoup}
Let $K$ be the transition kernel of a Markov chain with unique stationary distribution $\nu$ on state space $\Omega$. Let $\{X_{t} \}_{t \geq 0}$, $\{Y_{t} \}_{t \geq 0}$ be two Markov chains started at $X_{0} = x \in \Omega$ and $Y_{0} \sim \nu$. Define the coalescence time 
\be \label{DefCoalTime}
\tau(x) = \min \{ t \, : \, X_{t} = Y_{t} \}.
\ee Assume that the coupling of $\{X_{t} \}_{t \geq 0}$, $\{Y_{t} \}_{t \geq 0}$ satisfies $ X_{t} = Y_{t}$ for all $t \geq \tau(x)$. Then for any $ t \geq 0$,
\be 
\| \mathcal{L}(X_{t}) - \nu \|_{\TV} \leq \P[\tau(x) > t].
\ee 
\end{lemma}

The skeleton of the coupling argument in our paper is quite similar to that in \cite{Smit11a}. As in that paper, we construct a coupling of two copies of Kac's walk over two phases. In the first phase, we construct a `proportionate coupling' (see Definition \ref{def:propcoup}) between two copies of Kac's walk. We show that in this first phase, the two copies become close to each other in the Euclidean norm in  $O(n \log(n))$ time. In the second phase, also of length $O( n \log(n))$, we construct a non-Markovian coupling that allows two copies of Kac's walk that start very close to each other to actually coalesce. We then apply Lemma \ref{LemmaIneqCoup}, bounding the probability $\P[\tau(x) > T]$ that the random coalescence time $\tau(x)$ is larger than a particular deterministic time $T \approx 200 n \log(n)$.

This `two-step' approach, with an initial `contracting' phase followed by a second `coalescing' phase, is a popular approach for proving the convergence of Metropolis-Hastings chains on continuous state spaces (see \textit{e.g.,} \cite{roberts2002one, madras2010quantitative}). However, there are some difficulties in extending this approach to the study of high-dimensional Gibbs samplers such as Kac's walk. For most Metropolis-Hastings chains, two nearby Markov chains can coalesce in a single step, and so the coalescence phase can be of length 1. Gibbs samplers do not share this property: in $n$ dimensions, arbitrarily nearby Markov chains generally have 0 probability of coalescing in less than $n$ steps. This means that the coalescence phase is quite lengthy, and it becomes necessary to ensure that the two chains do not drift too far apart during this period. We also mention that the presence of constraints on the state space (in this case, the constraint that $X_{t}$ must be an element of $\sn$ and cannot be an arbitrary element of $\mathbb{R}^{n}$) presents additional complications that are not present for most Metropolis-Hastings chains. In our case, this constraint guarantees that any Markovian coupling scheme in the coalescence phase will have $\E[\tau(x)] > C \, n^{2}$ for some $C > 0$. The same phenomenon occurs in \cite{Smit11a} and many well-studied discrete chains (see \textit{e.g.,} Lemma 8 of \cite{bormashenko2011coupling}). See \cite{hayes2003non,hayes2007variable} for earlier examples of non-Markovian couplings applied to the study of Markov chains on discrete state spaces and \cite{connor2008optimal, kendall2015coupling} for more general discussion of when a `good' Markovian coupling exists.

\subsection{Structure of the Paper}
Section \ref{SecCont} describes phase 1 of the coupling, Section \ref{SecColl} describes phase 2, and the proof of Theorem \ref{MainThm} is given in Section \ref{SecMainThm}.

\subsection{Informal Discussion of `Greedy' Couplings}
Before giving a formal description of the second phase of our coupling in Section \ref{DefCoupCon}, we briefly explain where it comes from. 

The obvious `greedy' approach to coupling a pair of Gibbs samplers $\{W_{t},Z_{t}\}_{t \geq 0}$ in $\mathbb{R}^{n}$ is to run a Markovian coupling that matches as many coordinates as possible at each step. That is, we try to  grow the set $E_{t} = \{ 1 \leq i \leq n \, : \, W_{t}[i] = Z_{t}[i] \}$ by as much as possible at every step. Such a greedy coupling always exists, but it is easiest to analyze if $E_{t} \subset E_{t+1}$ for all $t$ - that is, if coordinates that agree once will agree forever. To phrase this second condition in a slightly unfamiliar way, this means that the sequence of hyperplanes 
\be \label{EqGreedyPlanes}
H_{t} = \{v \in \mathbb{R}^{n} \, : \, \forall \, i \in E_{t}, \, v[i] = 0 \} \ee
should satisfy
\be \label{IneqWish1}
H_{t+1} &\subset H_{t} \\
(W_{t} - Z_{t}) &\in H_{t}
\ee 
for all $t \geq 0$. When this occurs, the coalescence time $\tau(x)$ satisfies
\be \label{IneqWish2}
\tau(x) &= \inf \{t \geq 0 \, : \, H_{t} = \{ 0 \} \}. 
\ee

We claim that no Markovian coupling of the square $\{X_{t}^{2}, Y_{t}^{2} \}_{t \geq 0}$ of Kac's walk can satisfy Inequality \eqref{IneqWish1}. Indeed, for two copies of Kac's walk $X_t, Y_t$, the dynamics often \textit{force} $X_{t+1}^{2} - Y_{t+1}^{2}$ to `pop out' out of a candidate $H_{t}$.  However, our notation suggests a much larger family of `greedy' couplings: \textit{any} sequence $\{H_{t}\}_{t \geq 0}$ that satisfies Inequalities \eqref{IneqWish1} and \eqref{IneqWish2} gives rise to a `greedy' coupling of $\{Z_{t}, W_{t}\}_{t \geq 0}$, even if the hyperplanes aren't of the form \eqref{EqGreedyPlanes}.

There are many candidates for the sequence $\{H_{t}\}_{t \geq 0}$. Roughly speaking, in this paper we construct the simplest possible sequence $\{H_{t}\}_{t \geq 0}$ that is not of the form \eqref{EqGreedyPlanes}: rather than starting with $H_{0} = \mathbb{R}^{n}$ and greedily \textit{reducing} the dimension of our candidate $H_{t}$ as quickly as possible until it reaches $\{0\}$, which gives hyperplanes of the form \eqref{EqGreedyPlanes}, we start with $H_{T} = \{0\}$ at some fixed time $T >0$, and then go back in time while greedily \textit{increasing} the dimension of our candidate $H_{t}$ as quickly as possible until it reaches $\mathbb{R}^{n}$. This sequence of candidate hyperplanes, which are described carefully in Section \ref{DefCoupCon} (Equations \eqref{EqDefOfGoodChain} and \eqref{EqDefOfGoodChain2} define the sequence of hyperplanes used in this paper) turn out to satisfy \eqref{IneqWish1} and \eqref{IneqWish2} with high probability for Kac's walk. Because of this, the resulting non-Markovian coupling is almost as straightforward to analyze as the standard greedy Markovian coupling.

Our construction for $\{H_{t}\}_{t \geq 0}$ will not satisfy \eqref{IneqWish1} with high probability for all Markov chains, just as the usual construction in \eqref{EqGreedyPlanes} did not. Nonetheless, we hope that this general approach may yield other sophisticated couplings that are straightforward to analyze. We also point out that, when coupling Gibbs samplers, it is very common to force the two Markov chains to always select the same sequence of coordinates at every step\footnote{In the notation representation \eqref{KacStepRep}, this corresponds to choosing the same sequence $\{i_{t}, j_{t}\}_{t \geq 0}$.}, and we make this choice in the present paper. In this situation, the above coupling constructions would be applied to the Markov chains that are given after \textit{conditioning} on the sequence of coordinates to be updated. This modification makes no difference to the remainder of the discussion.

\section{Contraction Estimates} \label{SecCont}
 In this section, we couple two copies of the Kac's walk on $\sn$ and obtain a one-step contraction for the distance between these two copies in a suitable psuedo-metric. This is achieved via the following coupling.
 \begin{defn} [Proportional Coupling] \label{def:propcoup}
We define a coupling of two copies $\{X_t\}_{t \geq 0}, \{Y_t\}_{t \geq 0}$ of Kac's walk for a single step. Fix $X_{0},Y_{0} \in  \sn$. Let $(i_0,j_0,\theta_0)$ be the \textit{update variables} used by the random walk $X_{1}$ in representation \eqref{KacStepRep}, so that $X_{1} = \F(i_0,j_0,\theta_0,X_{0})$. Choose $\varphi \in [0, 2\pi)$ uniformly at random among all angles that satisfy  \footnote{If $X_{0}[i_{0}] = X_{0}[j_{0}] = 0$, all angles satisfy these equations. Otherwise, they have a unique solution, and the value of $\varphi-\theta_{0}$ modulo $2 \pi$ does not depend on $\theta_{0}.$ }
\be 
X_{1}[i_{0}] &= \sqrt{X_{0}[i_{0}]^{2} + X_{0}[j_{0}]^{2}} \cos (\varphi) \\
X_{1}[j_{0}] &= \sqrt{X_{0}[i_{0}]^{2} + X_{0}[j_{0}]^{2}} \sin (\varphi). 
\ee 
Then choose $\theta_{0}' \in [0, 2 \pi)$ uniformly among the angles that satisfy
\be 
\F(i_0,j_0,\theta_{0}',Y_{0})[i_{0}] &= \sqrt{Y_{0}[i_{0}]^{2} + Y_{0}[j_{0}]^{2}} \cos (\varphi) \\
\F(i_0,j_0,\theta_{0}',Y_{0})[j_{0}] &= \sqrt{Y_{0}[i_{0}]^{2} + Y_{0}[j_{0}]^{2}} \sin (\varphi)
\ee 
and set $Y_{1} = \F(i_0,j_0,\theta_{0}',Y_{0})$. 
\end{defn}
\begin{remark}
This coupling forces $Y_{1}$ to be as close as possible to $X_{1}$ in the Euclidean distance. For example, in dimension $n=2$, we always have $X_{1} = Y_{1}$ under this coupling. For $X_0, Y_0 \in \sn$, this coupling forces the three points $(0,0)$, $(X_{1}[i_{0}],X_{1}[j_{0}])$ and  $(Y_{1}[i_{0}],Y_{1}[j_{0}])$ to be collinear; in particular, it forces $X_{1}[i_{0}] Y_{1}[i_{0}], \, X_{1}[j_{0}] Y_{1}[j_{0}] \geq 0$.
\end{remark}

We set some notation for the remainder of the paper. Whenever we consider a pair of copies of Kac's walk $\{X_{t}\}_{t \geq 0}$, $\{Y_{t}\}_{t \geq 0}$, we define 
\be \label{EqABdef}
A_{t}[i] = X_{t}[i]^{2}, \qquad B_{t}[i] = Y_{t}[i]^{2}
\ee 
for all $t \geq 0$ and all $1 \leq i \leq n$.  For $x \in \mathbb{R}^n$, we define
\be
\|x\|_1 &= \sum_i |x[i]|, \quad \|x\|_2 = (\sum_i x[i]^2)^{1/2}.
\ee
Finally, we define $[n] = \{1,2,\ldots,n\}$.

\begin{lemma} \label{LemmaContEst}
Let $X_0, Y_{0} \in \sn$. For $t \geq 0$, couple $(X_{t+1}, Y_{t+1})$ conditional on $(X_{t},Y_{t})$ according to the coupling in Definition \ref{def:propcoup}. Then for any  $t \geq 0$, Kac's walk on $\sn$ satisfies
\be \label{IneqContractionConclusion}
\E[\sum_{i=1}^{n}(A_{t}[i] - B_{t}[i])^{2}] \leq  2(1 - \frac{1}{2n})^{t} .
\ee 
\end{lemma}

\begin{proof}
Fix $X_0, Y_{0} \in \sn$. For $t \geq 0$, couple $(X_{t+1}, Y_{t+1})$ conditional on $(X_{t},Y_{t})$ according to the coupling in Definition \ref{def:propcoup}, with update variables $i_{t},j_{t}, \theta_t, \theta_{t}'$ and additional variable $\varphi$ as in that definition.  We calculate
\be 
\E[ \sum_{k=1}^{n} &(A_{1}[k] - B_{1}[k])^{2}] \\
 &= \frac{2}{n(n-1)} \sum_{1 \leq i < j \leq n} \E[ \sum_{k=1}^{n} (A_{1}[k] - B_{1}[k])^{2} \, | \, (i_{0}, j_{0}) = (i,j)] \\
&= \frac{2}{n(n-1)} \frac{(n-1)(n-2)}{2} \sum_{k=1}^{n} ( A_{0}[k] - B_{0}[k])^{2}  \\
&\hspace{1cm}+ \frac{2}{n(n-1)} \sum_{i < j} \E[ (( A_{0}[i] + A_{0}[j]) \cos(\varphi)^{2} - ( B_{0}[i] + B_{0}[j]) \cos(\varphi)^{2} )^{2}]  \\
&\hspace{1cm}+ \frac{2}{n(n-1)} \sum_{i < j} \E[ (( A_{0}[i] + A_{0}[j]) \sin(\varphi)^{2} - ( B_{0}[i] + B_{0}[j]) \sin(\varphi)^{2} )^{2}] \\
&= \frac{n-2}{n} \sum_{k=1}^{n} ( A_{0}[k] - B_{0}[k])^{2} + \frac{4}{n(n-1)} \E[\cos(\varphi)^{4}] \sum_{i < j} (( A_{0}[i] + A_{0}[j])  - ( B_{0}[i] + B_{0}[j]) )^{2}\\
&= (1 -\frac{2}{n}) \sum_{k=1}^{n} ( A_{0}[k] - B_{0}[k])^{2} + \frac{3}{2n(n-1)} \sum_{i < j} (( A_{0}[i] + A_{0}[j])  - ( B_{0}[i] + B_{0}[j]) )^{2}\\
&=  (1 -\frac{2}{n}) \sum_{k=1}^{n} ( A_{0}[k] - B_{0}[k])^{2} \\
&\hspace{1cm}+ \frac{3}{2n(n-1)} \sum_{i < j} ( (A_{0}[i] - B_{0}[i])^{2} + (A_{0}[j] - B_{0}[j])^{2} + 2 (A_{0}[i] - B_{0}[i])(A_{0}[j] - B_{0}[j]) ) \\
&= (1 -\frac{2}{n}) \sum_{k=1}^{n} ( A_{0}[k] - B_{0}[k])^{2}\\
&\hspace{1cm} + \frac{3(n-1)}{2n(n-1)} \sum_{k=1}^{n} (A_{0}[k] - B_{0}[k])^{2} + \frac{3}{n(n-1)} \sum_{i < j}  (A_{0}[i] - B_{0}[i])(A_{0}[j] - B_{0}[j]) \\
&= (1 - \frac{1}{2n}) \sum_{k=1}^{n} ( A_{0}[k] - B_{0}[k])^{2} + \frac{3}{n(n-1)} \sum_{i < j}  (A_{0}[i] - B_{0}[i])(A_{0}[j] - B_{0}[j]) \\
&= (1 - \frac{1}{2n}) \sum_{k=1}^{n} ( A_{0}[k] - B_{0}[k])^{2} + \frac{3}{2n(n-1)} ( (\sum_{k=1}^{n}(A_{0}[k] - B_{0}[k]))^{2} - \sum_{k=1}^{n}(A_{0}[k] - B_{0}[k])^{2} ) \\
&= (1 - \frac{1}{2n} - \frac{3}{2n(n-1)})  \sum_{k=1}^{n} ( A_{0}[k] - B_{0}[k])^{2}. 
\ee 
Thus we have 
\be
\label{IneqMainContractionCalc}
\E[ \sum_{k=1}^{n} &(A_{1}[k] - B_{1}[k])^{2}] \leq (1 - {1 \over 2n}).
\ee
For $t \geq 0$, let $\mathcal{F}_{t} = \sigma( \cup_{0 \leq s \leq t}(X_{s},Y_{s}) )$ be the $\sigma$-algebra generated by the random variables $X_{0},\ldots,X_{t}$ and $Y_{0},\ldots,Y_{t}$. Repeatedly applying Equation \eqref{IneqMainContractionCalc}, we have for all $t \geq 0$ that
\be 
\E[\sum_{k=1}^{n} (A_{t}[k] - B_{t}[k])^{2} ] &= \E[ \E[ \sum_{k=1}^{n} (A_{t}[k] - B_{t}[k])^{2} | \mathcal{F}_{t-1}] ] \\
&\leq (1 - \frac{1}{2n}) \E[ \sum_{k=1}^{n} (A_{t-1}[k] - B_{t-1}[k])^{2} ] \\
&\leq (1 - \frac{1}{2n})^{t} \sum_{k=1}^{n} \E[(A_{0}[k] - B_{0}[k])^{2} ] \leq 2 (1 - \frac{1}{2n})^{t}.
\ee 
This completes the proof.
\end{proof}

\begin{remark}
We note that $X_{t}$ can be recovered exactly from $(A_{t}, \{ \frac{X_{t}[i]}{|X_{t}[i]|} \}_{i=1}^{n}) \in [0,1]^{n} \times \{-1,1\}^{n} \equiv \mathcal{X}$ by a map that we temporarily denote by $G \, : \, \mathcal{X} \mapsto \sn$, whose inverse exists except on a set of measure 0. Define the metric $d_{\mathcal{X}}$ on $\mathcal{X}$ by $d_{\mathcal{X}}((A,L), (A',L')) = \| A - A' \|_{2} + \sum_{i=1}^{n} \textbf{1}_{L[i] \neq L'[i]}$. Inequality \eqref{IneqContractionConclusion} can be combined with standard bounds on the contraction of simple random walk on the hypercube under the Hamming metric to show that, under the proportional coupling,
\be 
\E[ d_{\mathcal{X}}(G^{-1}(X_{t+1}), G^{-1}(Y_{t+1})) | X_{t},Y_{t}] \leq (1 - \frac{c}{n}) d_{\mathcal{X}}(G^{-1}(X_{t}), G^{-1}(Y_{t})) 
\ee  
for some $0 < c < \infty$. This contraction estimate, combined with Proposition 30 of \cite{ollivier2009ricci}, implies that the relaxation time of Kac's walk is $O(n)$. As shown in \cite{jan2001, carlen2003, maslen2003, caputo2008spectral}, this is the correct order for the relaxation time, and our argument provides a short alternative proof of the main result in \cite{jan2001}. However, we cannot use this approach to calculate the exact spectral gap, which was derived in \cite{maslen2003, carlen2003, caputo2008spectral}.
\end{remark}

We close this section with two elementary bounds. For $S \subset \{1,2,\dots,n\}$, let $\tau_{S} = \min \{ t \, : \, S \subset \cup_{0 \leq s \leq t } \{ i_{s}, j_{s} \} \, \}$. By the standard `coupon collector' bound, 
\be \label{IneqUseCouponCont}
\P[\tau_{[n]} > t] \leq n e^{-\frac{t}{n}}
\ee 
for all $ t \geq 0$. Let  $S = \{ 1 \leq i \leq n \, : \, X_{0}[i] \, Y_{0}[i] < 0\}$. By Definition \ref{def:propcoup}, we have $\min_{1 \leq i \leq n} X_{t}[i] \, Y_{t}[i] \geq 0$ for all $t \geq \tau_{S}$.
\begin{lemma} \label{LemmaSmallValues}
Let $Y \sim \mu$. Then for all $1 < c < \infty$ and any $1 \leq i \leq n$,
\be 
\P[Y[i] \leq n^{-3c}] \leq 2 n^{-c+1}.
\ee 
\end{lemma}
\begin{proof}
Let $\zeta_{1},\ldots,\zeta_{n}$ be $n$ i.i.d. random variables with $\mathcal{N}(0,1)$ distribution. Recall that the law $\mathcal{L}(Y[i])$ of $Y[i]$ satisfies $\mathcal{L}(Y[i]) = \mathcal{L}(\frac{\zeta_{1}^{2}}{\sum_{k=1}^{n} \zeta_{k}^{2}})$. By Markov's inequality and some direct computation,
\be \label{IneqContRep}
\P[Y[i] < n^{-3c}] &= \P[\frac{\zeta_{1}^{2}}{\sum_{k=1}^{n} \zeta_{k}^{2}} < n^{-3c}] \\
&\leq \P[\zeta_{1}^{2} \leq n^{-2c}] + \P[\sum_{k=1}^{n} \zeta_{k}^{2} \geq n^{c}]\\
& \leq  n^{-c} + n^{-c+1} \leq 2 n^{-c+1}
\ee   
and the proof is finished.
\end{proof}

\section{Non-Markovian Coupling and Coalesence Estimates} \label{SecColl}
In this section, we construct a more complicated non-Markovian coupling of two chains $\{X_{t}\}_{t \geq 0}$, $\{Y_{t} \}_{t \geq 0}$ started at a pair of near by points. In the proof of Theorem \ref{MainThm}, we will  show that this coupling leads to a collision of $X_t, Y_t$ with high probability for sufficiently large $t$. In order to couple the random walks $\{X_{t}\}_{t \geq 0}$, $\{Y_{t} \}_{t \geq 0}$, it is enough to couple the \textit{update sequences} $\{ (i_{t}(x), j_{t}(x), \theta_{t}(x)) \}_{t \geq 0}$, $\{ (i_{t}(y), j_{t}(y), \theta_{t}(y)) \}_{t \geq 0}$ used to update them in representation \eqref{KacStepRep}. Our coupling is as follows:

\subsection{Non-Markovian Coupling}\label{DefCoupCon}

Fix $T_{0} \leq T \in \mathbb{N}$. We construct the coupled update sequences 
$\{ (i_{t}(x), j_{t}(x), \theta_{t}(x)) \}$ and  $\{ (i_{t}(y), j_{t}(y), \theta_{t}(y)) \}$ for ${T_{0} \leq t < T}$ via the following procedure:

\begin{enumerate}
\item For each $T_{0} \leq t < T$, choose $1 \leq i_{t}(x) < j_{t}(x) \leq n$ uniformly at random and set $i_{t}(y) = i_{t}(x)$, $j_{t}(y) = j_{t}(x)$. 
\item Define a sequence of partitions $\{ \mathcal{P}_{t} \}_{t=T_{0}}^{T}$ of $\{1,2,\ldots,n\}$ inductively by the process: 
\begin{itemize} \label{DefProcPartition}
\item Set the partition $\mathcal{P}_{T} = \{ \{1\}, \{2\},\ldots,\{n\} \}$. 
\item Write $\mathcal{P}_{t+1} = \{P_{1}(t+1),\ldots,P_{\ell_{t+1}}(t+1) \}$ with $P_k(t+1) \subset \{1,2,\cdots, n\}$. Let $1 \leq u(t), v(t) \leq \ell_{t+1}$ be the indices satisfying $i_{t}(x) \in P_{u(t)}(t+1)$, $j_{t}(x) \in P_{v(t)}(t+1)$. If $u(t)=v(t)$, set $\mathcal{P}_{t} = \mathcal{P}_{t+1}$. If $u(t) \neq v(t)$, construct $\mathcal{P}_{t}$ by merging the sets $P_{u(t)}(t+1), P_{v(t)}(t+1)$. Thus, if $u(t) < v(t)$, set $\mathcal{P}_{t} = \{ P_{1}(t+1),\ldots,P_{u(t)-1}(t+1),P_{u(t)+1}(t+1),\ldots,P_{v(t)-1}(t+1),P_{v(t)+1}(t+1),\ldots,P_{\ell_{t+1}}(t+1), P_{u(t)}(t+1) \cup P_{v(t)}(t+1) \}$.
\end{itemize}
\item If $\mathcal{P}_{T_{0}} = \{1,2,\ldots,n\}$, continue the construction of the coupling by: 
\begin{itemize}
\item Define the ordered set 
\be 
\mathcal{S} = \{s_{1},\ldots, s_{n-1} \} \equiv \{T_{0} \leq t < T \, : \, \mathcal{P}_{t} \neq \mathcal{P}_{t+1} \}. 
\ee 
\item For $t \in \mathcal{S}^{c} \cap \{T_{0},T_{0}+1,\ldots,T-1\}$, choose the update variables $(\theta_{t}(x), \theta_{t}(y))$ according to the proportional coupling in Definition \ref{def:propcoup}.
\item Let $\Pi$ denote the set of all joint distributions on $[0,2\pi) \times [0,2\pi)$ with both marginal distributions uniformly distributed on $[0,2\pi)$. For $t = s_{k} \in \mathcal{S}$, choose \footnote{Note that there is generally not a unique coupling with this `maximal' property. In our analysis, we do not care which coupling with this property is used. A reader concerned about this choice can instead use a coupling of the form implicitly constructed in the proof of Lemma \ref{LemmaCoupSuccBound} at this step, without changing any of the bounds in the rest of the paper.} $(\theta_{s_{k}}(x), \theta_{s_{k}}(y)) \sim \pi \in \Pi$ so as to maximize the probability of the following events:
\be
\sum_{i \in P_{r}(t+1)} X_{t+1}[i]^{2} &= \sum_{i \in P_{r}(t+1)} Y_{t+1}[i]^{2}, \qquad 1 \leq r \leq \ell_{t+1}  \label{EqDefOfGoodChain} \\
X_{t+1}[k] Y_{t+1}[k] &\geq 0, \qquad k \in \{i_{t}(x), j_{t}(x) \}.\label{EqDefOfGoodChain2}
\ee  
\end{itemize}
\item If $\mathcal{P}_{T_{0}} \neq \{1, 2, \ldots,n\}$, continue the construction of the coupling by setting $\theta_{t}(y) = \theta_{t}(x)$ for all $T_{0} \leq t < T$.
\end{enumerate}

\begin{remark}
If \eqref{EqDefOfGoodChain} holds at time $T$, \eqref{EqDefOfGoodChain2} holds for all $T_{0} \leq t \leq T$, and $\mathcal{P}_{T_{0}} = \{1,2,\ldots,n\}$, then $X_{T} = Y_{T}$. Recall that forcing $X_{T}$ and $Y_{T}$ to coalesce is the goal of our coupling construction.
\end{remark}

\begin{remark}
By construction, we have that $\{ i_{t}(x), j_{t}(x) \}_{T_{0} \leq t < T}$, $\{ i_{t}(y), j_{t}(y) \}_{T_{0} \leq t < T}$ have the correct distributions. Going through the construction, it is easy to verify that the distribution of $\theta_{s}(x)$ (respectively $\theta_{s}(y)$) conditional on $\{ i_{t}(x), j_{t}(x) \}_{T_{0} \leq t < T}$ and $\{\theta_{t}(x)\}_{T_{0} \leq t < s}$ (respectively $\{ i_{t}(y), j_{t}(y) \}_{T_{0} \leq t < T}$ and $\{\theta_{t}(y)\}_{T_{0} \leq t < s}$) is uniform on $[0, 2\pi)$. Thus, Definition \ref{DefCoupCon} gives a valid coupling of two copies of Kac's walk.
\end{remark}

\subsection{Bounds for coupling}
The remainder of this section consists of some bounds that will allow us to prove that two chains $\{X_{t}\}_{T_{0} \leq t < T}$, $\{Y_{t}\}_{T_{0} \leq t < T}$ constructed according to the coupling described in Section \eqref{DefCoupCon} satisfy the conditions in 
Equations \eqref{EqDefOfGoodChain} and \eqref{EqDefOfGoodChain2} for all $T_{0} \leq t < T$ with high probability, as long as $T-T_{0}$ is sufficiently large and $\| X_{T_{0}} - Y_{T_{0}} \|$ is sufficiently small. 

We first bound the probability that the condition in step 3 of our coupling, $\mathcal{P}_{T_{0}} = \{1,2,\ldots,n\}$, fails to hold:

\begin{lemma} [Splitting of Partitions] \label{LemmaPartSplit}
Fix $\epsilon > 0$, $T_{0} \in \mathbb{N}$ and $T  > T_{0} + (\frac{1}{2} + 2 \epsilon) n \log(n)$. Then for $\mathcal{P}_{T_{0}}$ as in 
Section \ref{DefCoupCon} and all $n > N_{0}(\epsilon)$ sufficiently large,
\be 
\P[\mathcal{P}_{T_{0}} = \{1,2,\dots,n\}] \geq 1 - 2 n^{-\epsilon}.
\ee 
\end{lemma}

\begin{proof}
This follows immediately from Proposition 7.3 of \cite{Boll01}. Fix $M \geq 0$ and define the Erdos-Renyi graph $G_M$ on n vertices with exactly $M$ (possibly repeated) edges 
\be
E(G_m) = \cup_{s = T - M}^{T-1}(i_s,j_s).
\ee
Then
\be
\P[\mathcal{P}_{T_{0}} = \{1,2,\dots,n\}] = \P[G_{T-T_{0}} \text{ is connected}] \geq 1 - 2 n^{-\epsilon}
\ee 
where the last inequality follows from Proposition 7.3 of \cite{Boll01}.
\end{proof}

For $T_{0} \leq s < T$, define the event $\mathcal{A}(s)$ by 
\be \label{EqGoodCouplingSet}
\mathcal{A}(s) = \{ \text{Equations \eqref{EqDefOfGoodChain} and \eqref{EqDefOfGoodChain2} are satisfied for all } T_{0} \leq t \leq s. \}. 
\ee 
For $S \subset \{1,2,\dots,n\}$ and $X \in \mathbb{R}^{n}$, define 
\be 
\| X \|_{1,S} = \sum_{i \in S} | X[i] |.
\ee 
Also recall the definition of $A_{t}, B_{t}$ in Equation \eqref{EqABdef}. The following bound implies that good couplings will never drift too far apart:
\begin{lemma} [Closeness of Good Couplings] \label{LemmaCloseness}
Fix $T_{0} < T$, and couple two chains $\{X_{t}\}_{T_0 \leq t < T}$, $\{Y_{t}\}_{T_0 \leq t < T}$ according to the non-Markovian coupling defined in Section \ref{DefCoupCon}. Fix $T_0 \leq s \leq T$. Then, on the event $\mathcal{A}(s) \cap \{ \mathcal{P}_{T_0} = \{1, 2,\cdots, n\} \}$, we have 
\be \label{EqClosenessGoodCoupConc}
\| A_{t} - B_{t} \|_{1,S} \leq \| A_{T_0} - B_{T_0} \|_{1}
\ee 
for all $T_0 \leq t \leq s$ and all $S \in \mathcal{P}_{t}$. Furthermore, for
 all $T_0 \leq t \leq s$,
\be \label{eqn:AtBt1norm}
\|A_t - B_t\|_1 \leq n \|A_{T_0} - B_{T_0}\|_1. 
\ee
\end{lemma}

\begin{remark}
The final bound of Lemma \ref{LemmaCloseness}, Inequality \eqref{eqn:AtBt1norm}, is quite weak: it allows the distance $\|A_{t} - B_{t}\|_{1}$ to increase by as much as a factor of $n$. However, this additional factor of $n$ has a very small impact on our final bound in Theorem \ref{MainThm}; indeed, changing this \textit{multiplicative} factor of $n$ to a factor of $n^{k}$ would change our bound on the mixing time only by an \textit{additive} factor of $2 k n \log(n)$.

We also point out that, a priori, it is not obvious that \textit{any} polynomial bound holds. Indeed, the distance between $A_{t}$ and $B_{t}$ can double over a single step. Thus, naively, one obtains only a much larger bound, on the order of $n^{n}$.
\end{remark}

\begin{proof} [Proof of Lemma \ref{LemmaCloseness}]
We prove \eqref{EqClosenessGoodCoupConc} by induction on $t$. It is trivial at time $t = T_{0}$. Now assume that it holds at time $t$. We will show it also holds at time $t+1$ by considering two cases seperately:
\begin{enumerate}
\item \textbf{$t \notin \mathcal{S}$:} In this case, $\ell_{t} = \ell_{t+1}$ and $P_{k}(t) = P_{k}(t+1)$ for all $1 \leq k \leq \ell_{t}$. For all $k \neq u(t)$, it is immediate that $A_{t+1}[m] = A_{t}[m]$ and $B_{t+1}[m] = B_{t}[m]$ for all $m \in P_{k}(t)$, so by the induction hypothesis
\be 
\| A_{t+1} - B_{t+1}\|_{1,P_{k}(t+1)} = \| A_{t} - B_{t}\|_{1,P_{k}(t)} \leq \| A_{T_0} - B_{T_0} \|_{1}.
\ee 
For the remaining index $k = u(t)$, we calculate 
\be 
\| A_{t+1} - &B_{t+1} \|_{1,P_{u(t)}(t+1)} \\
&= \sum_{m \in P_{u(t)}(t+1)} | X_{t+1}[m]^{2} - Y_{t+1}[m]^{2} | \\
&= \sum_{m \in P_{u(t)}(t) \setminus \{ i_{t}, j_{t} \} } | X_{t}[m]^{2} - Y_{t}[m]^{2} | + | X_{t+1}[i_{t}]^{2} - Y_{t+1}[i_{t}]^{2} | + | X_{t+1}[j_{t}]^{2} - Y_{t+1}[j_{t}]^{2} | \\
&= \sum_{m \in P_{u(t)}(t) \setminus \{ i_{t}, j_{t} \} } | X_{t}[m]^{2} - Y_{t}[m]^{2} | + | (X_{t}[i_{t}]^{2} + X_{t}[j_{t}]^{2}) \cos(\varphi)^{2} - (Y_{t}[i_{t}]^{2} + Y_{t}[j_{t}]^{2}) \cos(\varphi)^{2} | \\
&\hspace{3cm}+ | (X_{t}[i_{t}]^{2} + X_{t}[j_{t}]^{2}) \sin(\varphi)^{2} - (Y_{t}[i_{t}]^{2} + Y_{t}[j_{t}]^{2}) \sin(\varphi)^{2} | \\
&= \sum_{m \in P_{u(t)}(t) \setminus \{ i_{t}, j_{t} \} } | X_{t}[m]^{2} - Y_{t}[m]^{2} | + | (X_{t}[i_{t}]^{2} + X_{t}[j_{t}]^{2})  - (Y_{t}[i_{t}]^{2} + Y_{t}[j_{t}]^{2})  |  \\
&\leq \sum_{m \in P_{u(t)}(t) \setminus \{ i_{t}, j_{t} \} } | X_{t}[m]^{2} - Y_{t}[m]^{2} | + | X_{t}[i_{t}]^{2}   - Y_{t}[i_{t}]^{2}   | + | X_{t}[j_{t}]^{2} -   Y_{t}[j_{t}]^{2} |\\
& = \| A_{t} - B_{t} \|_{1,P_{u(t)}(t)} \leq \| A_{T_0} - B_{T_0} \|_{1},
\ee     
where the last step is by the induction hypothesis. This completes the proof for the case $t \notin \mathcal{S}$.
\item \textbf{$t = s_{k} \in \mathcal{S}$:} By the same argument as in the previous case, we need only check that $\| A_{t+1} - B_{t+1}\|_{1,P_{u(t)}(t+1)}  \leq \|  A_{t} - B_{t}\|_{1,P_{u(t)}(t+1) \cup P_{v(t)}(t+1) }$ and $\| A_{t+1} - B_{t+1}\|_{1,P_{v(t)}(t+1)}  \leq \|  A_{t} - B_{t}\|_{1,P_{u(t)}(t+1) \cup P_{v(t)}(t+1) }$; the other $k-1$ expressions in equality \eqref{EqDefOfGoodChain} are satisfied automatically. By the symmetry of the problem, it is sufficient to check the first of these two inequalities. We compute 
\be 
\| A_{t+1} - B_{t+1}\|_{1,P_{u(t)}(t+1)} &= \sum_{m \in P_{u(t)}(t+1)} | A_{t+1}[m] - B_{t+1}[m] | \\
&= \sum_{m \in P_{u(t)}(t+1) \setminus \{ i_{t} \} } | A_{t+1}[m] - B_{t+1}[m] | + | A_{t+1}[i_{t}] - B_{t+1}[i_{t}] | \\
&= \sum_{m \in P_{u(t)}(t+1) \setminus \{ i_{t} \} } | A_{t}[m] - B_{t}[m] | + | A_{t}[i_{t}] + A_{t}[j_{t}]  - B_{t+1}[i_{t}] -  B_{t+1}[j_{t}]| \cos(\varphi)^{2} \\
&\leq \sum_{m \in P_{u(t)}(t+1) \cup \{ j_{t} \} } | A_{t}[m] - B_{t}[m] | \leq \|  A_{t} - B_{t}\|_{1,P_{u(t)}(t+1) \cup P_{v(t)}(t+1) } \\
&\leq \| A_{T_0} - B_{T_0} \|_{1},
\ee     
where the last step is by the induction hypothesis. 
\end{enumerate}
Thus, in either case, Equation \eqref{EqClosenessGoodCoupConc}  holds at time $t+1$, proving our first claim.\par Equation \eqref{eqn:AtBt1norm} follows by calculating
\be 
\| A_t - B_t \|_{1} &= \sum_{k=1}^{\ell_{t}} \| A_{t} - B_{t} \|_{1, P_{k}(t)} \\
&\leq  \sum_{k=1}^{\ell_{t}} \| A_{T_{0}} - B_{T_{0}} \|_{1} \\
&= \ell_{t} \| A_{T_{0}} - B_{T_{0}} \|_{1} \leq  n \| A_{T_{0}} - B_{T_{0}} \|_{1},
\ee 
where the first inequality follows from Equation \eqref{EqClosenessGoodCoupConc}, and we are done.
\end{proof}

The following Lemma will be used to show that it is possible to couple $X_{t+1},Y_{t+1}$ so as to satisfy Equations \eqref{EqDefOfGoodChain} and  \eqref{EqDefOfGoodChain2}, as long as $X_{t},Y_{t}$ satisfy those equations and have certain other good properties. 

\begin{lemma} [Coupling Success] \label{LemmaCoupSuccBound}
Fix positive reals $1 < p < q' < \frac{q}{2}$. Let $\theta, \theta' \sim \mathrm{Unif}[0,2\pi)$ and let $S = A + B \cos(\theta)^{2}$ and $S' = C + D \cos(\theta)^{2}$ for some $0 \leq A,B,C,D \leq 1$ that satisfy $|A-C|, |B-D| \leq n^{-q}$ and $ B,D \geq n^{-p}$. Then for $n > N_{0}(p,q, q')$ sufficiently large, there exists a coupling of $\theta, \theta'$ so that 
\be \label{IneqCoupSuccProb}
\P[S = S'] \geq 1 - 6 \times 10^{3} n^{- c}
\ee 
and
\be \label{IneqCoupSuccProb2}
\cos(\theta) \cos(\theta') \geq 0, \qquad \sin(\theta) \sin(\theta') \geq 0,
\ee 
where $c = \min(\frac{q'}{2}, q -2 q') > 0$.
\end{lemma}

\begin{remark}
In the notation of the construction in Section \ref{DefCoupCon}, we will use Lemma \ref{LemmaCoupSuccBound} with $A = (\sum_{k \in P_{t} \backslash \{i_{t}\} } X_{t}[k]^{2} )$ and $B= (X_{t}[i_{t}]^{2} + X_{t}[j_{t}]^{2})$. Roughly speaking, the upper bounds on $|A-C|, |B-D|$ in the statement of the lemma represent upper bounds on the distances between the medians and densities of two distributions that we wish to couple, and the lower bounds on $B,D$ represent upper bounds on how quickly the densities change. By inspection of Inequality \eqref{IneqStandardMinorizationBound} below, we expect to be able to couple two random variables with high probability if the distance between their medians and densities is small compared to the derivative of the density.

Note that we need both bounds to obtain the desired coupling inequality; as illustrated by the densities $\{ \frac{n^{10}}{2\sqrt{2 \pi}} e^{-n^{10}(x-n^{-1})^{2}} \}_{n \in \mathbb{N}}$, $\{ \frac{n^{10}}{2\sqrt{2 \pi}} e^{-n^{10}(x+n^{-1})^{2}} \}_{n \in \mathbb{N}}$, it is not enough to just have a bound on the distance between the medians.
\end{remark}

\begin{proof} [Proof of Lemma \ref{LemmaCoupSuccBound}]
We begin by showing that it is possible to satisfy inequality \eqref{IneqCoupSuccProb}.  Recall the standard inequality that, for any distributions $\nu_{1}, \nu_{2}$ with densities $f_{1}, f_{2}$ on $\mathbb{R}$,
\be \label{IneqStandardMinorizationBound}
\| \nu_{1} - \nu_{2} \|_{\mathrm{TV}} \leq 1 - \int \min(f_{1}(x), f_{2}(x)) dx. 
\ee 
The random variables $S, S'$ have densities 
\be  
f_{S}(x) &= \frac{1}{2 \pi} \frac{1}{\sqrt{B^{2} - 4(x-A-B/2)^{2}}}, \qquad \qquad x \in (A, A+B) \\
g_{S'}(x) &= \frac{1}{2 \pi} \frac{1}{\sqrt{D^{2} - 4(x-C-D/2)^{2}}}, \qquad \qquad x \in (C, C+D). \\
\ee 
Thus, applying Inequality \eqref{IneqStandardMinorizationBound} to these two densities with $1 < p < q' < \frac{q}{2}$ and $n > N_{0}(p,q,q')$ sufficiently large,
\be 
\|& S - S' \|_{\TV} \\
&= 1 - \int_{(A,A+B) \cap (C, C+D) } \min(f_{S}(x), g_{S'}(x)) dx \\
&\leq 1 - \frac{1}{2 \pi} \int_{ (A + n^{-q'},\,A+B -2 n^{-q'})  }\frac{1}{\sqrt{B^{2} - 4(x-A-B/2)^{2}}} \times \\
&\hspace{4cm} \min \big( 1 , { \sqrt{(D- n^{-q})^{2} - 4(|x-C-D/2| + 2n^{-q})^{2}} \over \sqrt{D^{2} - 4(x-C-D/2)^{2}}} \big)\, dx \\
&\leq 1 - \frac{1}{2 \pi} \int_{(A + n^{-q'},\,A+B -2 n^{-q'})} \frac{1}{\sqrt{B^{2} - 4(x-A-B/2)^{2}}} \times \\
 &\hspace{4cm} (1- 100 {n^{-q} \over D^2-4(x-C-D/2)^2})^{1 \over 2}  dx \\
&\leq 1 - (1- 8\times10^3 {n^{-q} \over n^{-2q'}})^{{1\over 2}} \frac{1}{2 \pi} \int_{ (A + n^{-q'},\,A+B -2 n^{-q'})} \frac{1}{\sqrt{B^{2} - 4(x-A-B/2)^{2}}} dx \\
&\leq  1 - (1- 8\times10^3 {n^{-q} \over n^{-2q'}})^{1 \over 2} (1-\frac{1}{2 \pi} \int_{(A ,\,A+2 n^{-q'}) \cup (A+B-2n^{-q'}, A+B)} \frac{1}{\sqrt{B^{2} - 4(x-A-B/2)^{2}}} dx ) \\
&\leq 1 - (1- 8\times10^3 {n^{-q} \over n^{-2q'}})^{1 \over 2}(1- {2 \over \pi} n^{-{1 \over 2} q'}) \\
&\leq 5 \times 10^3 n^{2q'-q}  + {2 \over \pi} n^{- {1 \over 2} q'}.
\ee 

This implicitly defines a coupling of $S,S'$ that satisfies $\P[S=S'] \geq 1 - 5 \times 10^3 n^{2q'-q}  + {2 \over \pi} n^{- {1 \over 2} q'} \geq 1 - 6 \times 10^3 \times n^{-c} $ for all sufficiently large $n$.\par 
Once we have a coupling of $S,S'$ satisfying \eqref{IneqCoupSuccProb}, we can extend it to a coupling of $\theta,\theta'$ that satisfies \eqref{IneqCoupSuccProb2} via a suitable rotation. To be more precise, define the map $G:[0,2\pi) \mapsto [0,1] \times \{-1,1\} \times \{-1,1\}$ by
\be 
G(\psi) \equiv (G_{1}(\psi), G_{2}(\psi), G_{3}(\psi)) = (\cos^{2}(\phi), \frac{\cos(\phi)}{|\cos(\phi)|}, \frac{\sin(\phi)}{|\sin(\phi)|}). 
\ee 
If $\theta \sim \mathrm{Unif}[0,2\pi)$, then $G_{1}(\theta)$, $G_{2}(\theta)$ and $G_{3}(\theta)$ are independent, with $\P[G_{2}(\theta) = 1] = \P[G_{3}(\theta) = 1] = \frac{1}{2}$. Furthermore, $S$ depends only on $G_{1}(\theta)$ and $S'$ depends only on $G_{1}(\theta')$. Thus, when extending a coupling of $S,S'$ to a coupling of $\theta, \theta'$ we are free to choose any coupling of $G_{2}(\theta), G_{2}(\theta')$ and any coupling of $G_{3}(\theta), G_{3}(\theta')$. We choose the coupling for which
\be 
G_{i}(\theta) = G_{i}(\theta')
\ee 
for $i \in \{2,3\}$. Thus, our coupling of $\theta,\theta'$ automatically satisfies both inequality \eqref{IneqCoupSuccProb} and \eqref{IneqCoupSuccProb2} whenever $S,S'$ satisfy inequality \eqref{IneqCoupSuccProb}.
\end{proof} 

\section{Proof of Theorem \ref{MainThm}} \label{SecMainThm}
We begin by proving the lower bound of Theorem \ref{MainThm}, 
Equation \eqref{EqMainThmLowerBound}. Fix $C_1 < 1/2$ and let $\{ T_1(n) \}_{n \geq 2}$ satisfy $T_{1}(n) < C_{1} n \log(n)$. Fix $X_{0} = x \in \sn$, let $\{ (i_{t},j_{t},\theta_{t} ) \}_{t \geq 0}$ be the update sequence associated with $\{X_{t} \}_{t \geq 0}$ and define
\be 
\zeta = \min \{t > 0 \, : \, \cup_{0 \leq s < t} \{ i_{s}, j_{s} \} = \{1,2,\ldots,n\} \}.
\ee 
For $t < \zeta$, $X_{t} \in \mathcal{A} \equiv \cup_{i=1}^{n} \{ X \in \sn \, : \, X[i] = x[i] \}$, and $\mu(\mathcal{A}) = 0$. Thus,
\be 
\| \mathcal{L}(X_{T_1(n)}) - \mu \|_{\TV} &\geq \P[X_{T(n)} \in \mathcal{A}] - \mu(\mathcal{A}) \\
&\geq \P[\zeta > T_1(n)] - 0. \label{EqLowerBoundCalc1}
\ee 
By the standard coupon collector problem (see, \textit{e.g.}, \cite{ErRe61}),
\be \label{EqLowerBoundCalc2}
\lim_{n \rightarrow \infty} \P[\zeta > T_1(n)] = 1.
\ee 
Combining inequality \eqref{EqLowerBoundCalc1} with \eqref{EqLowerBoundCalc2} completes the proof of Equation \eqref{EqMainThmLowerBound}.

Let $a = 47$, $b= 18.1$. To prove inequality \eqref{EqMainThmUpperBound}, we fix sequences $T_2(n)$, $T_{2}'(n)$ and $T_{2}''(n)$ that satisfy 
\be
200 n \log(n) < T_{2}(n)< n^2,
\ee
and  $T'_{2}(n) \geq (4a + 5) n \log(n)$, $T''_{2}(n) \geq  n \log(n)$ with  
\be
T'_{2}(n) + T''_{2}(n) = T_{2}(n)
\ee for all $n$ sufficiently large. We then construct a coupling of two copies $\{X_{t} \}_{t \geq 0}$, $\{Y_{t} \}_{t \geq 0}$ of Kac's walk on the sphere with starting points $X_{0} = x \in \sn$ and $\mathcal{L}(Y_{0}) = \mu$.  The coupling is as follows:
\begin{enumerate}
\item Couple $\{X_{t} \}_{0 \leq t \leq T'_{2}(n)}$, $\{Y_{t} \}_{0 \leq t \leq T'_{2}(n)}$ by using, at each step, the \textit{proportional coupling} as in Definition \ref{def:propcoup}.
\item Conditional on $\{X_{t} \}_{0 \leq t \leq T'_{2}(n)}$, $\{Y_{t} \}_{0 \leq t \leq T'_{2}(n)}$, couple $\{X_{t} \}_{T'_{2}(n) \leq t \leq T_2(n)}$, $\{Y_{t} \}_{T'_{2}(n) \leq t \leq T_2(n)}$ according to the non-Markovian coupling constructed in  Section \ref{DefCoupCon}. Thus in the notation of Section \ref{DefCoupCon}, we have $T_0 = T_2'(n)$ and $T = T_2(n)$.
\item Conditional on $\{X_{t} \}_{0 \leq t \leq T_{2}(n)}$, $\{Y_{t} \}_{0 \leq t \leq T_{2}(n)}$, we couple the remaining steps of the two chains as follows. First, run $\{X_{t} \}_{t> T_{2}(n)}$ conditional on $\{X_{t} \}_{0 \leq t \leq T_{2}(n)}$ according to its distribution. If $X_{T_{2}(n)} = Y_{T_{2}(n)}$, set $Y_{t} = X_{t}$ for all $t > T_{2}(n)$. If $X_{T_{2}(n)} \neq Y_{T_{2}(n)}$, run $\{Y_{t} \}_{t> T_{2}(n)}$ conditional on $\{Y_{t} \}_{0 \leq t \leq T_{2}(n)}$ according to its distribution and independently of $\{X_{t}\}_{t > T_{2}(n)}$.\footnote{The details of this third step of the coupling are irrelevant to the following analysis. They are provided only to give a concrete coupling that satisfies the condition $\P[ \forall t > \tau(x), \, X_{t} = Y_{t}] = 1$ in the statement of Lemma \ref{LemmaIneqCoup}.}

\end{enumerate} 
Define the events
\be
\mathcal{E}_{1} &= \{  \| A_{T'_{2}(n)} - B_{T'_{2}(n)} \|_{1} > n^{-a} \} \cup \{ \min_{1 \leq i \leq n} X_{T'_{2}(n)}[i] Y_{T'_{2}(n)}[i] < 0 \} \\
\mathcal{E}_{2} &= \{ X_{T_2(n)} \neq Y_{T_2(n)} \} \\
\mathcal{E}_{3} &= \{ \mathcal{P}_{T'_{2}(n)} \neq \{1,2,\ldots,n\}\}.
\ee
By Lemma \ref{LemmaIneqCoup}, 
\be 
\sup_{X_{0} = x \in \sn} \| \mathcal{L}(X_{T_2(n)}) - \mu \|_{\TV} &\leq \sup_{X_{0} = x \in \sn} \P[\mathcal{E}_{2}] \\
&\leq \sup_{X_{0} = x \in \sn} (\P[\mathcal{E}_{1}] + \P[\mathcal{E}_{3}] + \P[\mathcal{E}_{2} \cap \mathcal{E}_{1}^{c} \cap \mathcal{E}_{3}^{c}]).\label{IneqMainThmUpper1}
\ee 
The remainder of the proof of inequality  \eqref{EqMainThmUpperBound} will consist of bounding these three terms, in order.

Applying inequality \eqref{IneqUseCouponCont} and then Lemma \ref{LemmaContEst} and Markov's inequality, 
\be \label{IneqMainThmUpper2a}
 \lim_{n \rightarrow \infty} & \sup_{X_{0} = x \in \sn}\P[\mathcal{E}_{1}] \\
&\leq \lim_{n \rightarrow \infty}  \sup_{X_{0} = x \in \sn} (\P[ \| A_{T'_{2}(n)} - B_{T'_{2}(n)} \|_{1} \geq n^{-a}] + \P[\min_{1 \leq i \leq n} X_{T'_{2}(n)}[i] Y_{T'_{2}(n)}[i] < 0])\\
&\leq \lim_{n \rightarrow \infty}  \sup_{X_{0} = x \in \sn}( \P[ \| A_{T'_{2}(n)} - B_{T'_{2}(n)} \|_{2} \geq n^{-a-1}] + n e^{-\frac{T_{2}'(n)}{n}}) \\
&\leq \lim_{n \rightarrow \infty} (n^{2a + 2} (1 - \frac{1}{2n})^{T'_{2}(n)} + n^{-1}) = 0.
\ee 
By Lemma \ref{LemmaPartSplit}, 
\be \label{IneqMainThmUpper2b}
\lim_{n \rightarrow \infty}  \sup_{X_{0} = x \in \sn} \P[\mathcal{E}_{3}] = 0.
\ee 
Recall the event $\mathcal{A}(s)$ from Equation \eqref{EqGoodCouplingSet}:
\be 
\mathcal{A}(s) = \{ \text{Equations \eqref{EqDefOfGoodChain} and \eqref{EqDefOfGoodChain2} are satisfied for all } T'_{2}(n) \leq t \leq s. \}. 
\ee 
Our next goal is, roughly, to show that \textit{conditional on the first phase of the coupling having gone according to plan}, the set $\mathcal{A}(s+1) \backslash \mathcal{A}(s)$ has small probability; this is stated formally in Inequality \eqref{IneqMainIndStep1}. 

For $T_2'(n) \leq s < T_2(n)$, define
\be
\mathcal{B}(s) = \{ \min_{T_2'(n) \leq t < s} \min_{1 \leq i \leq n} Y_{t}[i]^{2} \geq \frac{1}{2} n^{-b} \}.
\ee 
We will bound $\P[\mathcal{E}_{2} \cap \mathcal{E}_{1}^{c} \cap \mathcal{E}_{3}^{c}]$ by showing that $\P[\mathcal{A}(t+1)^{c} \cap \mathcal{A}(t) \cap \mathcal{B}(t) \cap \mathcal{E}_{1}^{c} \cap \mathcal{E}_{3}^{c}]$ is small for all $T'_{2}(n) \leq t \leq T_{2}(n)$. \par
To this end, first note that
\be
\P[\mathcal{A}(T'_2(n))^{c} \cap \mathcal{E}_3^c] = 0.
\ee

Recall the set $\mathcal{S}$ from the construction of non-Markovian coupling from Section \ref{DefCoupCon}. We now consider two cases: $t \in \mathcal{S}$ and $t \in \{ T'_{2}(n),\ldots, T_2(n)\} \backslash \mathcal{S}$. In the case $t \in \{ T'_{2}(n),\ldots, T_2(n)\} \backslash \mathcal{S}$, we have
\be
\sum_{l \in P_r(t)} A_t[l] = \sum_{l \in P_r(t+1)} A_{t+1}[l],\,\quad P_{r}(t+1)= {P}_{r}(t) \in \mathcal{P}_t, \quad 1 \leq r \leq \ell_{t}.
\ee
This implies that
\be \label{IneqMainIndStep3} 
\P[\mathcal{A}(t+1)^{c} \cap \mathcal{A}(t) \cap \mathcal{B}(t) \cap \mathcal{E}_{1}^{c} \cap \mathcal{E}_{3}^{c}] \leq \P[\mathcal{A}(t+1)^{c} \cap \mathcal{A}(t) \cap \mathcal{E}_{3}^{c}]  = 0.
\ee 
Next, consider $t = s_{k} \in \mathcal{S}$. We will apply Lemma \ref{LemmaCoupSuccBound} with the random variables:
\be[EqDefSVars]
S &= \sum_{r \in P_{u(s_{k})}(s_{k}+1)\setminus \{i_{s_k}\}} X_{s_k}[r]^2 +  
(X_{s_k}[i_{s_k}]^2 + X_{s_k}[j_{s_k}]^2) \cos(\theta_{s_k})^2 \\
S' &= \sum_{r \in P_{u(s_{k})}(s_{k} + 1)\setminus \{i_{s_k}\}} Y_{s_k}[r]^2 +  
(Y_{s_k}[i_{s_k}]^2 + Y_{s_k}[j_{s_k}]^2) \cos(\theta_{s_k})^2.
\ee

On the event of $\mathcal{A}(s_k) \cap \mathcal{E}_1^c \cap \mathcal{E}_3^c$, Equation \eqref{eqn:AtBt1norm} of Lemma \ref{LemmaCloseness} gives that
\be \label{eqn:askbskt2}
\|A_{s_k} - B_{s_k}\|_1 \leq n\|A_{T_2'(n)} - B_{T_2'(n)}\|_1 \leq n^{1-a},
\ee
and so on the event $\mathcal{A}(s_{k}) \cap \mathcal{B}(s_{k}) \cap \mathcal{E}_{1}^{c} \cap \mathcal{E}_{3}^{c}$ we also have 
\be 
\min_{T_2'(n) \leq t < s} \min_{1 \leq i \leq n} \min( X_{t}[i]^{2}, Y_{t}[i]^{2}) \geq  n^{-b}.
\ee 
Thus, on the event $\mathcal{A}(s_{k}) \cap \mathcal{B}(s_{k}) \cap \mathcal{E}_{1}^{c} \cap \mathcal{E}_{3}^{c}$, the random variables $S,S'$ given in Equation \eqref{EqDefSVars} satisfy the assumptions of Lemma \ref{LemmaCoupSuccBound}, with $p=b$, $q= a-1$ and $q' = \frac{2(a-1)}{5} \in (b, \frac{a-1}{2})$. Thus, by  Lemma \ref{LemmaCoupSuccBound},
\be
\P\Big[\big \{ \sum_{i \in P_{u(s_{k})}(s_{k}+1)}& X_{s_{k}+1}[i]^{2} \neq \sum_{i \in P_{u(s_{k})}(s_{k}+1)} Y_{s_{k}+1}[i]^{2} \big \} \cap \mathcal{A}(s_{k}) \cap \mathcal{B}(s_{k}) \cap \mathcal{E}_{1}^{c} \cap \mathcal{E}_{3}^{c}\Big]\\
& \leq 6 \times 10^{3} n^{- \frac{2(a-1)}{5}}.  \label{IneqAlmostMainIndStep1}
\ee 
By the constraint $\sum_{i=1}^{n} A_{t+1}[i] = \sum_{i=1}^{n} B_{t+1}[i] = 1$, we have
\be 
\big \{ \sum_{i \in P_{u(s_{k})}(s_{k}+1)} X_{s_{k}+1}[i]^{2} \neq \sum_{i \in P_{u(s_{k})}(s_{k}+1)} Y_{s_{k}+1}[i]^{2} \big \} \cap \mathcal{A}(s_{k}) = \mathcal{A}(s_{k}+1)^{c}\cap \mathcal{A}(s_{k}).
\ee 
Combining this with inequality \eqref{IneqAlmostMainIndStep1} gives
\be \label{IneqMainIndStep1}
\P[\mathcal{A}(s_{k}+1)^{c} \cap \mathcal{A}(s_{k}) \cap \mathcal{B}(s_{k}) \cap \mathcal{E}_{1}^{c} \cap \mathcal{E}_{3}^{c} ] \leq 6 \times 10^{3} n^{- \frac{2(a-1)}{5}}.
\ee 

By Lemma \ref{LemmaSmallValues} and a union bound over times $T_2'(n) \leq t \leq T_{2}(n)$ and indices $1 \leq i \leq n$, we have for all sufficiently large $n$,
\be \label{IneqMainIndStep2}
\P[\mathcal{B}(T_2(n))^{c}] \leq 3 n^{4-\frac{b}{3}}.
\ee
Combining inequalities \eqref{IneqMainIndStep1} and \eqref{IneqMainIndStep3} and then \eqref{IneqMainIndStep2}, we have  
\be 
\P[\mathcal{E}_{2} \cap \mathcal{E}_{1}^{c} \cap \mathcal{E}_{3}^{c}] &= \P[ \mathcal{A}(T_2(n))^{c} \cap \mathcal{E}_{1}^{c} \cap \mathcal{E}_{3}^{c}] \\
&\leq \sum_{s=T_2(n)-T_2'(n)}^{T_2(n)} \P[ \mathcal{A}(T_2(n)-s+1)^{c} \cap \mathcal{A}(T_2(n)-s) \cap \mathcal{E}_{1}^{c}\cap \mathcal{E}_{3}^{c}]  + \P[\mathcal{A}(T_2'(n))^{c} \cap \mathcal{E}_{1}^{c} ] \\
&\leq \sum_{s=T_2(n)-T_2'(n)}^{T_2(n)} (\P[ \mathcal{A}(T_2(n)-s+1)^{c} \cap \mathcal{A}(T_2(n)-s) \cap \mathcal{B}(T_2(n)-s) \cap \mathcal{E}_{1}^{c} \cap \mathcal{E}_{3}^{c}]  \\
&\hspace{4cm}+ \P[\mathcal{B}(T_2(n) -s)^{c}]  )+ \P[\mathcal{A}(T_2'(n))^{c} \cap \mathcal{E}_{1}^{c}] \\
&\\
&\leq 6 \times 10^{3} \,T_2(n) \, n^{- \frac{2(a-1)}{5}}  + T_2(n)\, n^{4-\frac{b}{3}} + 0 \\
&\leq 6 \times 10^{3}  \, n^{2 - \frac{2(a-1)}{5}}   +  n^{6-\frac{b}{3}}.
\ee 
Since $b > 18$ and $a > 6$, this implies 
\be \label{IneqMainThmUpper3}
\lim_{n \rightarrow \infty}  \sup_{X_{0} = x \in \sn} \P[\mathcal{E}_{2} \cap \mathcal{E}_{1}^{c} \cap \mathcal{E}_{3}^{c}] = 0.
\ee 

Inequality \eqref{EqMainThmUpperBound} follows immediately from inequalities \eqref{IneqMainThmUpper1}, \eqref{IneqMainThmUpper2a}, \eqref{IneqMainThmUpper2b} and \eqref{IneqMainThmUpper3}. For sequences with $T_{2}(n) \leq n^2$ for all $n$ sufficiently large, this completes the proof of inequality \eqref{EqMainThmUpperBound}. For sequences with $T_{2}(n) > n^2$ infinitely often, inequality \eqref{EqMainThmUpperBound} follows from the case $T(n) = n^{2}$ by the fact that the total variation distance to stationarity of a Markov chain is monotonely decreasing in time. This completes the proof of Theorem \ref{MainThm}.

\section*{Acknowledgement}
NSP is partially supported by an ONR grant. AMS is partially supported by an NSERC grant.
\bibliographystyle{alpha}
\bibliography{KTCBib}

\end{document}